\documentclass{amsart}

\usepackage{amsmath}
\usepackage{url}
\usepackage{amsthm}
\usepackage{amsfonts}
\usepackage{hyperref}
\usepackage{color}
\usepackage{graphicx}

\theoremstyle{plain}
  \newtheorem{theorem}{Theorem}
  \newtheorem{lemma}{Lemma}
\theoremstyle{definition}

    \newtheorem{proposition}[subsection]{Proposition}

  \newcommand{\R}{\mathbb{R}}

\newcommand{\N}{\mathbb{N}}

\begin{document}
\title[]{Closed Mean Curvature Self-Shrinking Surfaces of Generalized Rotational Type}
\author[P.~McGrath]{Peter~McGrath}
\address{Department of Mathematics, Brown University, Providence,
RI 02912} \email{peter\_mcgrath@math.brown.edu}

\maketitle
\begin{abstract}
For each $n\geq 2$ we construct a new closed embedded mean curvature self-shrinking hypersurface in $\R^{2n}$.  These self-shrinkers are diffeomorphic to $S^{n-1}\times S^{n-1}\times S^1$ and are $SO(n)\times SO(n)$ invariant.  The method is inspired by constructions of Hsiang and these surfaces generalize self-shrinking ``tori'' diffeomorphic to $S^{n-1}\times S^1$ constructed by Angenent. 
\end{abstract}

\section{Introduction}
In the study of mean curvature flow, self-similar and in particular self-shrinking solutions arise naturally as separable solutions for the mean curvature flow PDE. 

An $n$-dimensional hypersurface $\Sigma^n \subset \R^{n+1}$ is called a \emph{mean curvature self-shrinker} (hereafter referred to as a \emph{self-shrinker}) if it satisfies the equation
\begin{equation}\label{shrinker equation}
 H = \frac{\langle \vec{X}, \vec{\nu}\rangle}{2}
 \end{equation}
 
where $\vec{X}$ is the position vector, and $\vec{\nu}$ is the unit normal such that $\vec{H} = - \vec{\nu}$.  Under this normalization, $\Sigma$ shrinks to a point at $t =1$. 

 Self-shrinking solutions are important in the analysis of singularities in mean curvature flow as a whole, as Huisken \cite{Huisken} has shown that the formation of a singularity in mean curvature flow resembles a self-shrinking solution after a sequence of appropriate rescalings.

Despite the importance of self-shrinkers in the study of singularities, the list of rigorously constructed examples of closed self-shrinkers is quite short.  
In 1989, Angenent \cite{Angenent} proved for each $n\geq 1$ the existence of an embedded mean curvature self-shrinker diffeomorphic to 
$S^{n-1}\times S^1$.  To the author's knowledge, the only known closed, embedded self-shrinkers other than the sphere and those constructed by Angenent were constructed by M\o ller in \cite{Moller}.  In the latter paper, the author constructs self-shrinkers with genus $g = 2k$ for each large enough $k\in \N$ in $\R^3$. His method involves desingularizing the intersection of a sphere and Angenent's torus in $\R^3$.

Within the class of rotationally invariant surfaces, the mean curvature flow equation reduces to a second-order nonlinear ODE on the space of orbits.  Kleene and M\o ller \cite{KM} conducted an analysis of the rotationally invariant case which resulted in a classification of complete embedded shrinkers in $\R^{n+1}$ invariant under $O(n)$.  We study mean curvature shrinkers with a more general rotational type, namely surfaces invariant under $O(m)\times O(n)$ for $m, n>1$.  A primary motivation is to construct new examples of closed, embedded self-shrinkers.  

The ansatz of rotational invariance as a mechanism for constructing objects satisfying some geometric property is not new.  Consider for example the constant mean curvature constructed by Delaunay in 1841 \cite{Delaunay}.  The Delaunay surfaces were the first nontrivial examples of constant mean curvature surfaces and have been used as building blocks for more recent constructions \cite{Kapouleas CMC}.   Beginning in the 1960s, Hsiang began a systematic study of manifolds invariant under general Lie groups. In a joint work with Lawson \cite{H4}, the authors classify minimal surfaces in $S^n$ with certain invariance groups of  ``low cohomogeneity."  Subsequently Hsiang used these methods which he described as ``equivariant differential geometry'' to prove various results, in particular the existence of minimal hyper-spheres in $S^{n}$ not congruent to the equator for various $n\geq 4$ (the so-called ``spherical Bernstein problem'' \cite{H1}), and the existence of infinitely many noncongruent, closed, embedded minimal surfaces in $S^n$ for $n\geq 3$ \cite{H3}.

In the spirit of Angenent's construction of self shrinking surfaces diffeomorphic to $S^{n-1}\times S^1$ and invariant under $SO(n)$, we more generally prove

\begin{theorem}\label{theorem: construction}
For each integer $n\geq 2$, there is an embedded mean curvature self-shrinker $\Sigma^{2n-1}\subset \R^{2n}$ diffeomorphic to $S^{n-1}\times S^{n-1}\times S^1$ and invariant under the action of $SO(n)\times SO(n)$ on $\R^{2n}$.

\end{theorem}

The self-shrinkers constructed in Theorem \ref{theorem: construction} are the first rigorously constructed family of odd-dimensional closed embedded self-shrinkers since Angenent's examples from 1989.  

The author extends his thanks to his thesis advisor, Nikolaos Kapouleas, for suggesting Hsiang's techniques as an avenue for constructing self-shrinking surfaces and for pointing him to Angenent's paper.  It is a pleasure to thank Justin Corvino and Frederick Fong for extensive feedback on an earlier version of this paper.  The author also thanks Farhan Abedin and Mamikon Gulian for helpful conversations.

%%%%%%%%%%%%%%%%%%%%%%%%%%%%%%%%%%%%%%%%%%%%%%%%%%%%%%%%%%%%%%%%
% SECTION 1: NOTATION AND TERMINOLOGY
%
%%%%%%%%%%%%%%%%%%%%%%%%%%%%%%%%%%%%%%%%%%%%%%%%%%%%%%%%%%%%%%%%
\section{Notation and Terminology}

Let $O(m)\times O(n)$ act on \[ \R^{m+n}= \R^m\oplus \R^n = \{ (\vec{x}, \vec{y}): \vec{x}\in \R^m, \vec{y}\in \R^n\} \]  with the usual product action.  Let $i: \Sigma^{m+n-1} \hookrightarrow \R^{m+n}$ be an immersed hypersurface.  We will abuse notation by identifying $\Sigma$ with its image in $\R^{m+n}$.  We say a hypersurface $\Sigma^{m+n-1}$ is \emph{invariant} under the action of $O(m)\times O(n)$ if the action of $O(m)\times O(n)$ on $\R^{m+n}$ preserves $\Sigma$, so in particular $\Sigma$ has a foliation by copies of $S^{m-1}\times S^{n-1}$ of varying radii.
We identify the orbit space $\R^{m}\times \R^n / (O(m)\times O(n))$ with the closed first quadrant $Q = \{ (x, y)\in\R^2: x, y \geq 0\}$ under the projection map $\Pi: \R^{m+n}\rightarrow Q$ defined by  \[ \Pi (\vec{x}, \vec{y})=( |\vec{x}|, |\vec{y}|) := (x, y).\]  We call the image $\Pi(\Sigma)$ in the orbit space of an $O(m)\times O(n)$-invariant self-shrinker $\Sigma$ the associated \emph{profile curve} $\gamma_\Sigma$.  Up to isometries, there is a one to one correspondence between smooth $O(m)\times O(n)$ invariant hypersurfaces in $\R^{m+n}$ and smooth curves in the interior of $Q$.
We denote $S^k(r) = \{ |x| \in \R^{k+1}: |x| = r\}$ the sphere of radius $r$ in $\R^k$.  For a submanifold $\Sigma \subset S^{k-1}(1) \subset \R^k$ the \emph{cone over} $\Sigma$ is the set
$ C(\Sigma) = \{ p \in \R^k: p/|p| \in \Sigma\}.$

Since the class of self-shrinkers in $\R^k$ are minimal surfaces with respect to the metric
\[ e^{- \frac{|X|^2}{2k}} \sum_{i=1}^{k} (dx^i)^2\]
(where $X$ denotes the position vector in $\R^k$) which is conformal to the standard metric by a Gaussian factor \cite{Angenent},
it follows that the profile curve associated to an $O(m)\times O(n)$ invariant shrinker is a geodesic with respect to the metric

\begin{align}\label{eq: gaussian metric}
 g = x^{2(m-1)} y^{2(n-1)} e^{-\frac{(x^2+y^2)}{2}} \left( dx^2+ dy^2\right)
\end{align}
which degenerates along the $x$ and $y$ axes of $Q$.  The Euler-Lagrange equation for the length functional corresponding to this metric is an ODE whose solution curves correspond to self-shrinkers, but we choose to derive this ODE by directly computing the principle curvatures below.

Let $\Sigma$ be an $O(m)\times O(n)$ invariant hypersurface with profile curve $\gamma_\Sigma(t) = (x(t), y(t))$, where $\gamma_\Sigma$ is parametrized with respect to Euclidean arc-length.   By the rotational invariance, one finds $\Sigma$ has $m-1$ principle curvatures equal to 
\[ \frac{ y'(t)}{x(t) (x'(t)^2+y'(t)^2)^\frac{1}{2}},\]
$n-1$ principle curvatures equal to
\[ -\frac{ x'(t)}{y(t) (x'(t)^2+y'(t)^2)^\frac{1}{2}}\]
and one principle curvature equal to
\[ \frac{x'(t)y''(t)-x''(t)y'(t)}{ (x'(t)^2+y'(t)^2)^\frac{3}{2}}.\] 
The unit normal $\nu(t)$ to $\gamma(t)$ is \[ \frac{\left(-y'(t), x'(t)\right)}{(x'(t)^2+y'(t)^2)^\frac{1}{2}},\] hence

\[\frac{\langle \vec{X}, \vec{\nu}\rangle}{2} = \frac{1}{2} \frac{y(t)x'(t) - y'(t)x(t)}{( x'(t)^2+y'(t)^2)^{\frac{1}{2}}}.\] 
Therefore the self-shrinker Equation \eqref{shrinker equation} reduces in this case to

\begin{multline}
\label{parametric ode}
 -x''(t) y'(t) + x'(t) y''(t) = \\
 \left( \frac{x(t) y'(t) - x'(t) y(t)}{2} + \frac{ (n-1) x'(t)}{ y(t)} - \frac{(m-1)y'(t)}{x(t)} \right) (x'(t)^2+ y'(t)^2).
\end{multline}

If we assume $\gamma_\Sigma$ is locally graphical over the $x$-axis, $y = u(x)$, then Equation \eqref{parametric ode} reduces to

\begin{align}\label{graphical ode}
u''(x) = \left( \frac{ xu'(x) - u(x)}{2} +\frac{(n-1)}{u(x)} -\frac{(m-1) u'(x)}{x}\right) (1+(u'(x))^2).
\end{align}

If we introduce $\theta(t) = \arctan \left( \frac{y'(t)}{x'(t)}\right)$ and compute $\theta'(t)$ via Equation \eqref{parametric ode}, we get the system

\begin{align}\label{geodesic flow}
\begin{cases}
\dot{x} &= \cos \theta \\
\dot{y} &= \sin \theta \\
\dot{\theta} &= \left( \frac{x}{2} - \frac{m-1}{x}\right) \sin \theta + \left( \frac{n-1}{y} - \frac{y}{2}\right) \cos \theta
\end{cases}
\end{align}
which is a flow on the unit tangent bundle of $Q$.  Observe that equation \eqref{geodesic flow} remains true even at places where $x'(t) = 0$.  This system will be useful in describing the local behavior of solutions of Equation \eqref{parametric ode}.

Below, we will say that a curve $\gamma(t)$ is a \emph{geodesic} or a \emph{solution} if it solves equation \eqref{parametric ode}.  With an appropriate parametrization, such a curve is actually a geodesic with respect to the metric given in \eqref{eq: gaussian metric}. If $\gamma$ is further unit speed parametrized (with respect to the usual Euclidean metric on $Q$), we will also sometimes call $\gamma(t)$ a \emph{geodesic} or a \emph{solution} if the triple $(x(t), y(t), \theta(t))$ satisfies Equation \eqref{geodesic flow}.

We will denote the graph of the line $y = \sqrt{\frac{n-1}{m-1}}$ by $\ell$, and let $\ell^+$ and $\ell^-$ be the regions in $Q$ where $y> \sqrt{\frac{n-1}{m-1}}$ and $y<\sqrt{\frac{n-1}{m-1}}$ respectively.  
%%%%%%%%%%%%%%%%%%%%%%%%%%%%%%%%%%%%%%%%%%%%%%%%%%
%  SECTION I: ODE ANALYSIS 
%%%%%%%%%%%%%%%%%%%%%%%%%%%%%%%%%%%%%%%%%%%%%%%%%%
\section{ODE Analysis}

We first catalogue some trivially verified solutions to equation \ref{parametric ode} which will be useful in the proofs of Theorem 1.  

\begin{proposition}\label{prop: simple solutions}
The following curves are solutions to $(1)$:
\begin{enumerate}
	\item $y = \sqrt{ \frac{ n-1}{m-1} } \,x$
	\item $x^2+y^2 =2(m+n-1)$
	\item $x = \sqrt{2(m-1)}$;  $x = 0$
	\item $y= \sqrt{2(n-1)}$;  $y = 0$.
\end{enumerate} 
Moreover, these are the unique solutions among the following classes of curves: lines through the origin, circles centered at the origin, vertical lines, horizontal lines.
\end{proposition}

Geometrically, the self-shrinker corresponding to the line $y= \sqrt{ \frac{ n-1}{m-1} } x$ is the cone over the product $S^{m-1}\left( \sqrt{ \frac{ m-1}{m+n -2}}\right)\times S^{n-1} \left( \sqrt{ \frac{ n-1}{m+n-2}}\right)$.  It is straightforward to see that $S^{m-1}\left( \sqrt{ \frac{ m-1}{m+n -2}}\right )\times S^{n-1} \left( \sqrt{ \frac{ n-1}{m+n-2}}\right)$ is minimal in $S^{m+n-1}$, and we more generally have the following trivial result.

\begin{proposition}
Suppose $\Sigma^{k-2} \subset S^{k-1}$ is a minimal surface.  Then $C(\Sigma)$, the cone over $\Sigma$, satisfies the self-shrinker equation in $\R^k$. 
\end{proposition}
\begin{proof}
It is a well known fact that $C(\Sigma)\subset \R^k$ is minimal if $\Sigma \subset S^{k-1}$ is minimal, so $H( C(\Sigma))=0$.  Since $C(\Sigma)$ is a cone, $\langle \vec{X}, \vec{\nu} \rangle = 0$. 
\end{proof}

The portion of the circle $x^2+y^2 = 2(m+n-1)$ in $Q$ corresponds to the sphere $S^{m+n-1}(\sqrt{2(m+n-1)})$ as foliated by copies of $S^{m-1}\times S^{n-1}$.
The lines $x = \sqrt{2(m-1)}$ and $y = \sqrt{2(n-1)}$ correspond to the products $S^{m-1}(\sqrt{2(m-1)})\times C( S^{n-1})$ and $C(S^{m-1})\times S^{n-1}(\sqrt{2(n-1)})$.  Finally, the axes $x = 0$ and $y=0$ correspond to the cones $C(S^{n-1})$ and $C(S^{m-1})$.  These are not hypersurfaces, hence will not be discussed further.  

 Although the Equations \eqref{parametric ode} and \eqref{geodesic flow} are singular when $x = 0$ or $y = 0$, Proposition \ref{prop: simple solutions} nevertheless exhibits geodesics which intersect these lines, albeit orthogonally.  A straightforward modification of the proof of part 2 of  Lemma 9 in \cite{KM} proves that this is the \emph{only} way a geodesic may intersect one of the axes.

\begin{lemma}\label{lemma: degeneracy}
Suppose $\gamma_i: (a, b), i=1, 2$ are solutions of equation \eqref{parametric ode}   and 
\begin{enumerate}
\item $\gamma_1$ is a graph over the $x$-axis and  $\lim_{t\rightarrow b} = (x_b, 0)$ where $x_b> 0$.
\item $\gamma_2$ is a graph over the $y$-axis and  $\lim_{t\rightarrow b}  = (0, y_b)$ where $y_b>0$. 
\end{enumerate}
Then $\gamma_i$ extends smoothly to $(a, b]$ and $\gamma_i$ intersects the corresponding axis orthogonally. 
\end{lemma}

The following lemma places some coarse restrictions on the behavior of geodesics of Equation \eqref{parametric ode}.  Part (2) is analogous to lemma 8 of \cite{KM}.  
\begin{lemma}\label{lemma: critical points}
Let $\gamma(t) = (x(t), y(t))$ be a solution of \ref{geodesic flow}. 
\begin{enumerate}
\item Unless $\gamma$ is either $x = \sqrt{2(m-1)}$ or $y=\sqrt{2(n-1)}$, any critical point of $x(t)$ or $y(t)$ is either a strict local minimum or maximum.
\item The functions $y(t) - \sqrt{2(n-1)}$ and $x(t) - \sqrt{2(n-1)}$ have neither positive minima nor negative maxima, and these functions have different signs at successive critical points.
\item Suppose $\gamma(t_0) \in \ell^-$ and $\dot{\theta}(t_0)>0, \dot{x}(t_0)>0, \dot{y}(t_0)>0$.  For any $t$ in the maximal interval containing $t_0$ on which $\gamma$ lies in $\ell^-$ and $x(t), y(t)$ remain monotone, $\dot{\theta}(t)>0$.  An analogous statement is true for $\ell^+$.  
\end{enumerate}  
    \end{lemma}
\begin{proof}
The first two statements follow from inspecting the system
\begin{align}
\begin{cases}
\dot{x} &= \cos \theta \\
\dot{y} &= \sin \theta \\
\dot{\theta} &= \left( \frac{x^2 - 2(m-1)}{2x}\right) \sin \theta + \left( \frac{2(n-1) - y^2}{2y}\right) \cos \theta.
\end{cases}
\end{align}
For the third, we compute from the above system

\begin{align*}
\ddot{\theta} &= \dot{x}\dot{y} \left( \frac{m-1}{x^2} - \frac{n-1}{y(t)^2}\right) + \dot{\theta}\left( \frac{x^2 - 2(m-1)}{2x} \cos\theta +  \frac{y^2 - 2(m-1)}{2y} \sin\theta\right).
\end{align*}  
In particular, when $\dot{\theta} = 0$, $\ddot{\theta} = \dot{x}\dot{y}\left( \frac{m-1}{x^2} - \frac{n-1}{y(t)^2}\right)$.  Hence if $\gamma(t)\in \ell^-$ and $\dot{\theta}(t) = 0$, then $\ddot{\theta}(t)>0$.  This implies $(3)$.  
\end{proof}

When $m=n$ will be convenient to consider geodesics that can be written locally as normal graphs over the line $\ell$ and so we introduce the following rotated coordinates. 

\begin{align*}
r(t) = \frac{1}{\sqrt{2}}\left(x(t)+y(t)\right), && s(t) &= \frac{1}{\sqrt{2}} \left(x(t)-y(t)\right).
\end{align*}

Then defining $\phi = \arctan \left( \frac{r'(t)}{s'(t)}\right)$, we find (when $m=n$) the flow on the unit tangent bundle in these coordinates becomes

\begin{align}\label{sr flow}
\begin{cases}
\dot{r} &= \cos \phi  \\
\dot{s} &= \sin \phi \\
\dot{\phi} &= \left( \frac{s}{2} + \frac{n-1}{r^2 - s^2} s\right) \sin \phi + \left( \frac{n-1}{r^2 - s^2} r - \frac{r}{2} \right) \cos \phi.
\end{cases}
\end{align}

Next we characterize smooth geodesics which pass through the origin.

\begin{proposition}\label{degeneracy at origin}
Any solution of Equation \eqref{parametric ode} which intersects $(0,0)$ is the line $y  = \sqrt{\frac{ n-1}{ m-1}}\; x$.
\end{proposition}
\begin{proof}
We prove the proof in the case that $m=n$ (which will be the only case we actually use) and leave the routine modifications for the general case to the reader.
Let $\gamma(t): [0, t_m]$ be a geodesic parametrized such that $\gamma(t_m) = (0, 0)$.  By Lemma \eqref{lemma: critical points}  when $0\leq x < \sqrt{2(n-1)}$ and $0\leq y < \sqrt{2(n-1)}$ the only critical points $x(t)$ and $y(t)$ may have are minima.  Moreover, since $\gamma(t_m) = (0, 0)$, it follows that $\dot{x}(t)<0, \dot{y}(t)< 0$ for $t$ sufficiently close enough to $t_m$, and we assume without loss of generality that $\dot{x}(t), \dot{y}(t)< 0$ for all $t\in [0, t_m]$.  It follows that $\theta(t) \in (- \pi/2, -\pi)$ for all $t\in [0, t_m]$.    These facts imply $\gamma$ is graphical over the $x$ and $y$ axes, and since both $x(t)$ and $y(t)$ are monotonically decreasing, $\gamma$ can also be written as a normal graph $s = f(r)$ over the line $\ell$.  We will show that $\lim_{t\nearrow t_m} \theta(t) = \frac{-3\pi}{4}$, so that by uniqueness $\gamma$ coincides with $\ell$.  We split the remainder of the argument into two cases.

Case 1: $\gamma(t)$ eventually remains on one side of $\ell$.  Without loss of generality, we suppose that $x(t)\geq y(t)$ for $t\in [0, t_m]$.  By inspection of the system \eqref{sr flow}, we conclude that $f$ has at most one critical point which must be a maximum.  By Lemma \ref{lemma: critical points} part (3),  $\theta(t)$ is eventually monotonic and so $\lim_{t \nearrow t_m} \theta(t)$  exists.  If $\theta(t_m) = -\frac{3\pi}{4}$, then $\gamma$ is the line $\ell$, so first suppose $\theta(t_m) <- \frac{3\pi}{4}$ and $\dot{\theta}<0$, so $\gamma$ lies below $\ell$.  Then $y(t) \leq \tan (\theta(t_m)) x(t)$ and so 
\begin{align*}
\dot{\theta}(t) &= \frac{x(t)}{2}\sin\theta(t) - \frac{y(t)}{2}\cos\theta(t) + (n-1)\left( \frac{1}{y(t)} - \frac{1}{x(t)}\right)\\
&>  \frac{x(t)}{2}\sin\theta(t) - \frac{y(t)}{2}\cos\theta(t) + (n-1)\frac{x'(t)}{2}\left( \frac{1}{\tan \theta(t_m)} - 1\right) \frac{1}{x(t)}.
\end{align*}
 
By assumption, $1-\frac{1}{\tan\theta(t_m)}> 0$, so  after using the fact that $|\dot x|, |\dot y| \leq 1$ and integrating this inequality, we conclude that for any $t_2 >t_1$, we have

\[ \theta(t_2) - \theta(t_1)>  O(1) + (n-1)\log\left[\frac{1}{2}\left(\frac{1}{\tan\theta(t_m)} - 1\right) \frac{ x(t_1)}{x(t_2)}\right].\]
But since $\theta(t_2) - \theta(t_1)$ is bounded, it follows that $x(t_2)$ is bounded away from $0$ as $t_2\rightarrow t_m$, a contradiction.  The other cases, where $\dot{\theta}>0$ and $\gamma$ is above $\ell$, are similar.  Hence, it must be that $\lim_{t\nearrow t_m}\theta(t)= - \frac{3\pi}{4}$.  

Case 2: $\gamma(t)$ intersects $\ell$ infinitely many times.  By compactness, there is a convergent sequence $t_k\rightarrow t_\infty$ of places where $\gamma(t_k)$ lies on $\ell$.  If $t_\infty < t_m$, then because $\gamma$ is analytic where it is smooth, it is immediate that $\gamma$ coincides with $\ell$.  Hence we may suppose that $t_k\rightarrow t_m$.  We claim that $\lim_{k\rightarrow \infty} \theta(t_k) = -\frac{3\pi}{4}$.  To see this, observe that 
\[ \dot{\theta}(t_k) = \frac{n-1}{x(t_k)} \left( \sin\theta(t_k) - \cos \theta(t_k)\right) + O(x). \]
If it is not the case that $\theta(t_k)\rightarrow -\frac{3\pi}{4}$, the preceding equation implies $\dot{\theta}(t_k)$ becomes unbounded as $k\rightarrow \infty$.  In this case, it is straightforward to see from the system \eqref{geodesic flow} that for a sufficiently large $K$, $\gamma$ will fail to be graphical over $\ell$ near $t_K$, a contradiction.
Hence $\lim_{k\rightarrow \infty} \theta(t_k) = \frac{-3\pi}{4}$ and so the image of $\gamma$ is contained in $\ell$.
\end{proof}

%%%%%%%%%%%%%%%%%%%%%%%%%%%%%%%%%%%%%%%%%%%%%%%%%%
% Construction of closed geodesics when m = n
%%%%%%%%%%%%%%%%%%%%%%%%%%%%%%%%%%%%%%%%%%%%%%%%%%
\section{Construction of a closed embedded geodesic when $m = n$} 

In this section, we prove Theorem \ref{theorem: construction}.  Throughout, we assume that $m=n>1$.  

When $m=n$, the metric 
\begin{align}
\label{geodesic metric}
 g = x^{2(n-1)} y^{2(n-1)} e^{-\frac{(x^2+y^2)}{2}} \{ dx^2+ dy^2\}
\end{align}
is preserved by reflection through the line $\ell$.  If follows from this that the set of geodesics of $g$ is also preserved under reflection through $\ell$.

We will prove the existence of a geodesic $\gamma_{R_*}$ with an embedded segment which lies on one side of $\ell$ and intersects $\ell$ orthogonally two times.  To do this, we will adapt the argument of Angenent \cite{Angenent} to this setting.  Reflecting the said geodesic segment through $\ell$ shows that $\gamma_{R_*}$ is a closed embedded geodesic in $Q$.  Under the identification between geodesics of the metric \eqref{geodesic metric} and $O(n)\times O(n)$ invariant self-shrinkers described in Section 1, $\gamma_{R_*}$  corresponds to a closed embedded $O(n)\times O(n)$ invariant self-shrinker and Theorem \ref{theorem: construction} will follow.  

We require that $m=n$ so we may use the preceding reflection argument.  It is probable that when $m\neq n$ a closed embedded geodesic intersecting the line $\ell$ exists, although a different method would be needed for the proof.

\begin{proof} (of Theorem \eqref{theorem: construction}).  

We shall only consider parts of geodesics $\gamma(t) = (r(t), s(t))$ for which $s(t)\geq 0$, in other words, parts which lie below $\ell$.  
Define $\gamma_R = (r_R(t), s_R(t))$ to be the solution of Equation \eqref{sr flow} with initial conditions $\gamma_R(0) = (0, R)$ and $\phi_R(0) = 0$. 
Inspection of \eqref{sr flow} shows that $\phi_R'(0)< 0$, so near $0$, $\gamma_R$ can be written as a graph $f_R: [r_R(t_m(R)), R]\rightarrow [0, \infty)$ of a non-negative function over a maximal connected interval $[r_R(t_m(R)), R]$.  In particular, $t_m(R)$ satisfies either $\phi_R(t_m(R)) = 0, \phi_R(t_m(R)) = -\pi, s_R(t_m) = 0$ or $s_R(t_m) = 0$.  At a critical point of $f_R$, $\phi_R$ is $-\pi/2$, so Equation \eqref{sr flow} implies that $\phi'_R <0$ there.  Then by definition, $f_R$ has at most one critical point, which (if it exists) must be a local maximum.  

\begin{lemma}\label{lemma: angle turn}
For $R$ sufficiently large, there is $s_0(R) = R- O(\frac{1}{R})$ so that $f_R$ attains a maximum at $s_m$.  Furthermore, $f_R(s_0) = R - O(\frac{1}{R})$.
\end{lemma}
\begin{proof}
Define a rescaled time variable $\tau$ by $\tau = R t$, so in particular $\frac{dt}{d\tau} = \frac{1}{R}$.  Then from Equation \eqref{sr flow} we see
\begin{align}\label{eqn: angle tau}
\frac{d\phi}{d\tau} &= \frac{1}{R} \left( \frac{s(\tau)}{2}+\frac{s(\tau)}{r^2(\tau)-s^2(\tau)}\right)\sin(\phi(\tau))
- \frac{r}{R}\left( \frac{1}{2} - \frac{1}{r^2(\tau) - s^2(\tau)}\right)\cos(\phi(\tau)).
 \end{align}
Given any $\epsilon >0$ and $C>0$ and $0<\tau < C$, we can pick $R$ large enough that $\frac{r(\tau)}{R} >(1-\epsilon)$. Then by estimating \eqref{eqn: angle tau} It follows that for such $\tau$

\[ \frac{d\phi}{d\tau} < - \frac{1-\epsilon}{2} \cos (\phi(\tau)).\]

The equation 
\[  \frac{d\phi}{d\tau} = - \frac{1-\epsilon}{2} \cos (\phi(\tau))\]
has explicit solution
\[ \phi(\tau) = - 2 \arctan \left( \tanh\left( \frac{(1-\epsilon)t}{4}\right)\right)\]
and it is straightforward to see that
\begin{align}\label{eqn: angle estimate}
 - 2 \arctan \left( \tanh\left( \frac{(1-\epsilon)t}{4}\right)\right) +\frac{\pi}{2} = O(e^{-\frac{(1-\epsilon)t}{2}}).
 \end{align}
For some fixed small $\tau_0>0$ it is not hard to see that as long as $\frac{d\phi}{d\tau}<0$ and $\tau>\tau_0$, there is a constant $c_1$ such that $s(\tau)>\frac{c_1}{R}$.
Then combining \eqref{eqn: angle tau} and \eqref{eqn: angle estimate}, it follows that there is a $\tau<C$ such that $\phi(\tau)  = -\frac{\pi}{2}$.
\end{proof}

By combining Lemma \ref{lemma: critical points}, part (3) and Lemma \eqref{lemma: angle turn}, we conclude 
\begin{lemma}\label{lemma: monotonicity}
For large $R$, $\phi_R$ is monotonic at least until $\gamma$ crosses $\ell$ or $\theta_R(t) = -\pi$. 
\end{lemma}

\begin{lemma}\label{lemma: uniform convergence}
Let $\delta_n\searrow 0$. There is either a sufficiently large $R$ so that $f_R(r_R(t_m)) = 0$ or there is a sequence $R_n\nearrow \infty$ such that $f_{R_n}$ is defined on an interval which contains $(\delta_n, R)$. 
\end{lemma}
\begin{proof}
First we show $f_R$ cannot end on the $x$-axis.  By Lemma \ref{lemma: degeneracy}, if $\gamma_R(t_1)$ lies on the $x$-axis, $\theta_R(t_1) = -\frac{\pi}{2}$.  This is not possible since we know that $\theta_R$ is decreasing at least until a critical point of $y_R$, where $\theta_R = -\pi$. By continuity, if there were a later time when $\theta_R = - \frac{\pi}{2}$, there would be a second critical point of $f_R$, which is impossible.  By Lemma \ref{lemma: monotonicity}, $\dot{\phi}_R< 0$ as long as $\phi_R > -\frac{3\pi}{4}$.   Next we show that for some fixed large $R_0$, which depends only on $n$, we have
\begin{enumerate}
\item $f_R(R_0) = O(\frac{1}{R})$
\item $\phi_R(R_0) = O(\frac{\log R_0}{R})$.
\end{enumerate}
The first item follows trivially since $f_R = O(\frac{1}{R})$ at its maximum.  Now define $\alpha_R(t) = \phi_R(t)+\frac{\pi}{2}$.  Then since $\cot \phi = -\tan\alpha$,

\begin{align*}
\frac{d\alpha}{dr} = \frac{d\phi}{dr} = \frac{d\phi}{dt}\frac{dt}{dr} = \frac{ \left( \frac{s}{2} + \frac{n-1}{r^2 - s^2} s\right) \sin \phi + \left( \frac{n-1}{r^2 - s^2} r - \frac{r}{2} \right) \cos \phi}{\sin \phi}\\
= \left( \frac{s}{2} + \frac{n-1}{r^2 - s^2} s\right) + \left( \frac{r}{2}-\frac{n-1}{r^2 - s^2} r  \right) \tan\alpha.
\end{align*}
Thus when \[ \alpha = \arctan \left (-\frac{ \frac{s}{2} + \frac{(n-1)s}{r^2-s^2}}{\frac{r}{2} - \frac{(n-1)r}{r^2-s^2}}\right) = O\left(-\frac{1}{Rr}\right)\]  (when $r$ is large and $s$ is small) one has $\frac{d\alpha}{dr} = \frac{d\phi}{dr} = 0$.  However, by the above, we know that $\frac{d\phi}{dr}<0$ as long as $\phi$ is not too small, so for large $r$, $\frac{d\phi}{dr}$ is $O(\frac{1}{rR})$. By integrating, the second claim follows.
From the two claims and the smooth dependence of ODE solutions on initial conditions, $f_R$ converges to $\ell$ in $C^1$ on compact subsets.  Hence given $\delta_n$, there is an $R_n$ such that $\gamma_{R_n}$ remains graphical over $\ell$ at least until $\delta_n$.  Hence, if $\gamma_{R_n}$ does not cross $\ell$, $f_{R_n}$ is defined on an interval containing $[\delta_n, R]$.  
\end{proof}

\begin{lemma}\label{lemma: hits r axis}
For sufficiently large $R$, $s_R(t_1) = 0$ and $r_R(t_1)\rightarrow 0$.  
\end{lemma}
\begin{proof}
Suppose according to the conclusion of lemma \ref{lemma: uniform convergence} that there are sequences $\delta_n \searrow 0$ and $R_n \nearrow \infty$ and functions $0< f_{R_n}(r)< \frac{C}{R_n}$ defined on $(\delta_n, R_n)$ where each $f_{R_n}$ satisfies the equation 
\begin{equation}\label{graphical ode}
\frac{ f''}{1+f'^2} = \left( \frac{ (n-1)r}{r^2 - f(r)^2} - \frac{r}{2}\right) f'(r) + \left( \frac{1}{2} +\frac{n-1}{r^2 - f(r)^2}\right) f = 0.
\end{equation}
By Lemma \eqref{lemma: uniform convergence} it follows that $f_{R_n}(r)$ and $f'_{R_n}(r)\rightarrow 0$ uniformly on compact sets as $n\rightarrow \infty$. We now note that there is a constant $C>0$ such that $ | f_{R_n}'(1)| \leq C | f_{R_n}(1)|$ for every $n$.  Indeed, no such $C$ exists, since $f'_{R_n}(1)>0$, the Mean Value Theorem $f_{R_n}$ must intersect the $r$ axis or become nongraphical for some $r> \frac{1}{R_n}$, a contradiction.  
 Define $g_{R_n}(r)$ to be the rescaling
\[ g_{R_n}(r) = \frac{ f_{R_n}(r)}{f_{R_n}(1)}. \] 
By combining Equation \eqref{graphical ode} and the bound $|f'_{R_n}(1)|< C|f_{R_n}(1)|$, the $f_{R_n}$ have uniform $C^2$ bounds on compact subintervals.  Therefore, 
by the Ascoli-Arzela Theorem, there is a subsequence of the $g_{R_n}$ (which for convenience of notation, we take to be the original sequence) such that $g_{R_n}$ converges to $g$ in $C^2$ on compact subsets of $(0, \infty)$.  
Since $f_{R_n}\rightarrow 0$, the limit $g$ of the rescalings is a solution of the linearization of  Equation \eqref{graphical ode} about the zero solution.  Since $g_{R_n}(1)  = 1$ and $g_{R_n}> 0$ for each $n$, $g$ is a positive solution of

\[ g''(r) = \left( \frac{n-1}{r} - \frac{r}{2} \right)g' + \left(  \frac{1}{2} + \frac{n-1}{r^2} \right) g   = 0.\]
Furthermore, by Lemmas \eqref{lemma: angle turn}, \eqref{lemma: monotonicity} and \eqref{lemma: uniform convergence}, $g'(r)\geq 0$ for all $x\in (0, \infty)$, so $\lim_{r\searrow 0} g(r)$ exists and is finite.
Set $h(r) = e^{-\frac{r^2}{8}}g(r)$.  $h(r)$ is also positive on $(0, \infty)$, $\lim_{r\searrow 0} h(r) =\lim_{r\searrow 0} g(r)$ exists, and moreover $h$ satisfies the equation

\begin{align}\label{eqn: modified ode}
 h''+ \frac{n-1}{r} h' + \left( \frac{n}{4}- \frac{r^2}{16}+\frac{1}{2} + \frac{n-1}{r^2}\right) h = 0.
 \end{align} 
This equation has a regular singularity at $r = 0$, so using a method of Frobenius (see for instance \cite{Bender}), one can write the solution space of \eqref{eqn: modified ode} as the span of two solutions $\{ x^{\alpha_1}A(x), x^{\alpha_2}B(x)\} $ for some $\alpha_1, \alpha_2 \neq 0$ and $A(x), B(x)$ analytic.  More specifically, the $\alpha_i$ are the roots of  $x^2 + (n-2)x + (n-1) = 0$, that is
\begin{align*}
\alpha_i = \frac{ - (n-2) \pm \sqrt{ (n-2)^2 - 4(n-1)}}{2}.
\end{align*}

It is easy to see that $\sqrt{(n-2)^2 - 4(n-1)}$ is never an integer when $n\in \mathbb{N}$ is greater than $1$, so the solution set is spanned by $\{ r^{\alpha_1}A(r), r^{\alpha_2}B(r)\}$ for some $A(r), B(r)$ analytic in a neighborhood of $0$.  

 Thus, when $2\leq n <7$, solutions to the linearized equation have an oscillatory behavior near $0$ and hence fail to be strictly positive. When $n\geq 7$, any nonzero solution of \eqref{eqn: modified ode} has a singularity like $x^{-\frac{n-2}{2}}$ near $0$.  Hence $\lim_{r\searrow 0} h(r)$ fails to be finite, a contradiction. 
\end{proof}

By lemma \eqref{lemma: hits r axis}, $\gamma_R(t_1)$ lies on $\ell$ when $R$ is large enough.  However, by Proposition \eqref{prop: simple solutions}, when $R = \sqrt{2(2n-1)}$ $f_R$ begins and ends on the $x$-axis.  Hence
\[ R_* = \inf\{ R>0 : f_{\tilde{R}}(r_{\tilde{R}}(t_m)) = 0 \text{ for all } \tilde{R}>R\}\]
is well defined and greater than $0$.  

\begin{figure}[t]
\includegraphics[width=.5\textwidth]{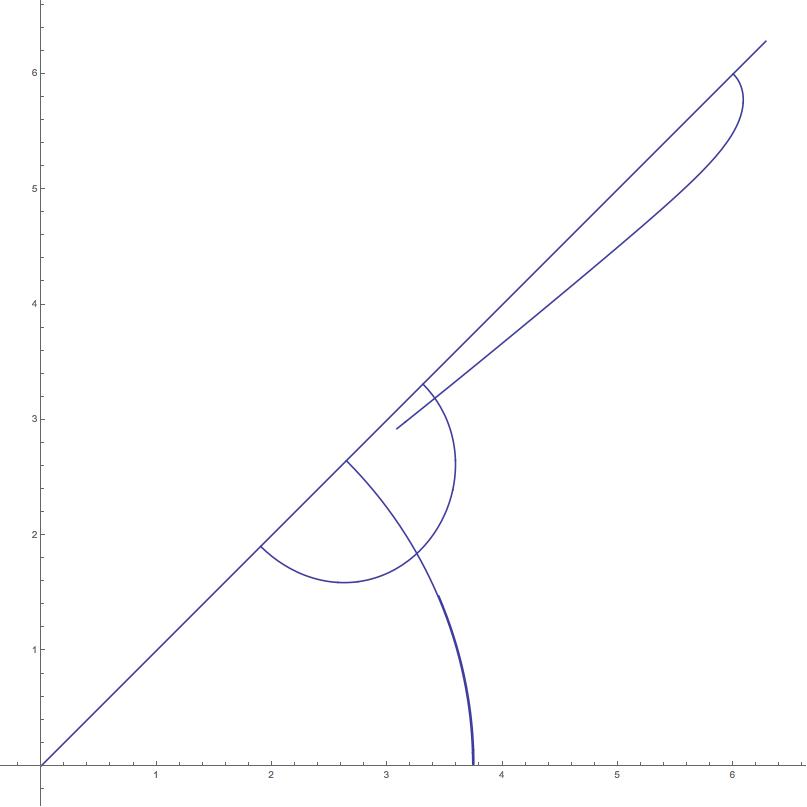}

\caption{The line $\ell$ and pieces of curves $\gamma_R$ for three different values of $R$ (when $m=n=4$.)  For large $R$, $\gamma_R$ behaves as in Lemma \ref{lemma: angle turn}. For $R=R_{*}$, $\gamma_{R_*}$ intersects $\ell$ orthogonally twice, and for $R = \sqrt{2(2n-1)}, \gamma_R$ is the arc of a circle and intersects the $x$-axis. } 
\end{figure}

\begin{lemma} $R_*$ satisfies
\begin{enumerate}
\item $\liminf_{R\searrow R_*}r_R(t_m) > 0$.
\item $\liminf_{R\searrow R_*} y_R(t_m) > 0$.
\end{enumerate}
\end{lemma}
\begin{proof}
We prove both statements by contradiction.  For (1), suppose there is a sequence $R_n\searrow R_*$ with $r_{R_n}(t_1(R_n)) \rightarrow 0$.  It follows that $\gamma_{R_n}$ converges in $C^2$ on compact sets to a geodesic $\gamma$ passing through the origin.  By proposition \ref{degeneracy at origin}, $\gamma$ is the line $\ell$.  But $\gamma_{R_n}$ does not converge in $C^2$ to $\ell$, since in particular $\theta_{R_n}(0) = -\frac{\pi}{4}$.

For (2), if there were a sequence $R_n \searrow R_*$ and times $t_{R_n}$ with $y_{R_n}(t_{R_n}) \rightarrow 0$, by compactness, we could find a subsequence (which we assume to be the original sequence) such that $\gamma_{R_n}(t_{R_n}) \rightarrow (x_*, 0)$ for some $x_* \geq 0$.  By part (1) of this lemma, we may assume that $x_*> 0$.  Then $\gamma_{R_n}$ converges to a geodesic $\gamma$ which intersects the $x$ axis at $(x_*, 0)$.  By Lemma \eqref{lemma: degeneracy}, $\gamma$ intersects the $x$-axis orthogonally.  Since $\gamma_{R_n}$ smoothly converges to $\gamma$ on compact sets away from $(x_*, 0)$, there are $\delta_n, \epsilon \rightarrow 0$ and times $t_n$ such that $\gamma_{R_n}(t_n) = (x_*+\delta_n, \epsilon_n)$ and $\theta_{R_n}(t_n) = -\frac{\pi}{2}$.  Since $\dot{\theta}(t_n) = -O(x_*)>0$ independent of $n$.  Hence, shortly after $t_n$, $\dot{\theta} = -O(x_*)-O(\frac{1}{\epsilon_n})$.  Hence, for large $n$, near $(x_*, 0)$, $\gamma_{R_n}$ travels nearly vertically downward, makes a sharp bend near $(x_*, 0)$, and then travels nearly vertically upward.  In particular, for sufficiently large $n$, $\gamma_{R_n}$ fails to be a normal graph over $\ell$ on an interval strictly smaller than $[0, t_{R_n}]$.  This contradicts the definition of $t_{R_n}$.  
\end{proof}

\begin{proposition}
$\gamma_{R_*}$ begins and ends on $\ell$. Moreover $\phi_{R_*}(t_m) = -\pi$.  
\end{proposition}
\begin{proof}
The preceding lemma shows that as $R\searrow R_*$ the $\gamma_{R}$ are contained in a compact subset of $Q$ which is disjoint from the $x$ and $y$ axes.  Hence, smooth dependence on initial conditions shows that $\gamma_{R_*}$ starts and ends on $\ell$.  By the definition of $R_*$, it follows that $\theta_{R_*} \geq -\pi$.  Since the set $\{ R: \phi_R(t_m)>-\pi \}$ is open by smoothness on initial conditions, it must be that $\phi_{R_*}(t_m) = -\pi$ for otherwise, we would contradict the minimality of $R_*$.  

\end{proof}

Since $\gamma_{R_*}$ intersects $\ell$ orthogonally, the union of $\gamma_{R_*}$ with its reflection it through $\ell$ is a smooth embedded closed geodesic $\gamma$.  Under the identification between geodesics of \eqref{geodesic metric} and $O(n)\times O(m)$ invariant self-shrinkers in Section 2, $\gamma$ corresponds to an embedded closed $O(n)\times O(n)$ shrinker and Theorem \ref{theorem: construction} follows.  
\end{proof}

\section{Final Remarks}

\begin{figure}[t]
\label{fig:immersed}
\includegraphics[width=.5\textwidth]{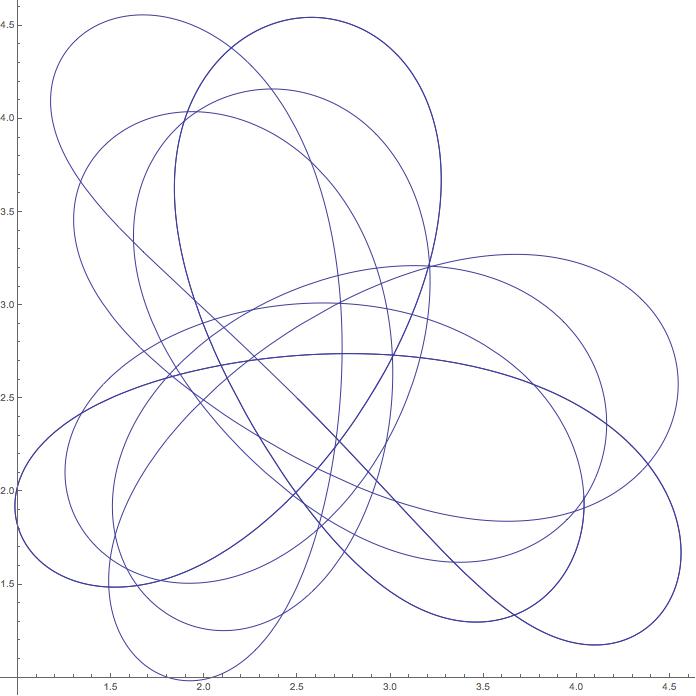}
\caption{Plot in Mathematica of a geodesic which is likely closed and immersed (when $m=n=4$).} 
\end{figure}

In \cite{DK}, the authors construct a large number of \emph{immersed} closed self-shrinkers with a single rotational symmetry.  It seems likely analogous techniques apply to the setting of this paper.  It is easy to find numerical evidence for such closed immersed examples; see Figure 2.

We conclude the paper with an observation (compare with Theorem 4, part (1) in \cite{KM}) which places some restrictions on the behavior of embedded geodesics.  

\begin{proposition}
Any embedded closed solution of \eqref{parametric ode} intersects $\ell$ at least twice.
\end{proposition}
\begin{proof}
We argue by contradiction.  There is clearly no closed geodesic which intersects $\ell$ exactly once, since the intersection would have to be tangential which would contradict uniqueness.  Hence we may suppose $\gamma$ is a closed geodesic lying below $\ell$.  Compactness implies the distance between $\gamma$ and $\ell$ is greater than $0$, hence there is a smallest $c>0$ such that the curve $\hat{\gamma}:= \gamma(t) - c {e}_1$ (where ${e}_1$ is the standard basis vector in $\R^2$ pointing in the positive $x$ direction) makes its first point of contact $p := (x_0, y_0)$ with $\ell$ at a time $t_0$.  Near $p$, if $\hat{\gamma}(t) = (x_{\hat{\gamma}}(t), y_{\hat{\gamma}}(t), \theta_{\hat{\gamma}}(t))$ is parametrized so that $\dot{x}_{\hat{\gamma}}(t)>0$, one has $\dot{\theta}_{\hat{\gamma}}(t_0)\leq 0$, since otherwise $\hat{\gamma}$ would go above $\ell$.    On the other hand, by using Equation \eqref{geodesic flow}, we have that 
\begin{align*}
- \dot{\theta}_{\hat{\gamma}}(t_0)=\dot{\theta}_{\ell}(t_0) - \dot{\theta}_{\hat{\gamma}}(t_0) &= \left( \left(\frac{x_0}{2} - \frac{n-1}{x_0}\right) - \left(\frac{x_1}{2} - \frac{n-1}{x_1}\right)\right)\sin\theta(t_0)
\end{align*}
where $x_1 = x_0+c$. 

It is easily checked that $h$ is monotone increasing.  Since $\sin\theta>0$ and $x_0$ and $x_1$ are the $x$ coordinates of the points on $p$ and $p+c{e}_1$, this implies the right hand side of the above equation is $<0$, which implies $\dot{\theta}_{\hat{\gamma}}(t_0)>0$, which is a contradiction.
\end{proof}

%%%%%%%%%%%%%%%%%%%%%%%%%%%%%%%%%%%%%%%%%%%%%%%%%%
% BIBLIOGRAPHY 
%%%%%%%%%%%%%%%%%%%%%%%%%%%%%%%%%%%%%%%%%%%%%%%%%%

\end{document}